\documentclass[a4paper]{amsart}
\usepackage{amsmath, amsfonts, amsthm, amssymb}
\author{Matthew Towers}
\title{Hochschild cohomology of $U(\ssl_2)$}
\email{m.towers@ucl.ac.uk}
\date{\today}

\newcommand{\ext}{\operatorname{Ext}}
\newcommand{\ssl}{\mathfrak{sl}}

\newcommand{\im}{\operatorname{im}}
\newcommand{\HH}{\operatorname{HH}}
\renewcommand{\gg}{\mathfrak{g}}

\theoremstyle{plain}
\newtheorem{thm}{Theorem}[section]
\newtheorem{lem}[thm]{Lemma}
\newtheorem{cor}[thm]{Corollary}
\newtheorem{prop}[thm]{Proposition}

\begin{document}

\begin{abstract}
Let $k$ be an algebraically closed field of characteristic $p>2$. We
determine the Hochschild cohomology of $U(\ssl_2(k))$ and the invariants
of $\ssl_2(k)$ and $\mathsf{SL}_2(k)$ in the adjoint action on the
divided power algebra $D(\ssl_2(k))$.
\end{abstract}

\maketitle

\section{Introduction}
Let $\gg$ be the Lie algebra of a semisimple algebraic group $G$ over a
field of characteristic $p>0$. There are several associative algebras
associated to $\gg$: the universal enveloping algebra $U(\gg)$ and its
commutative analogue the symmetric algebra $S(\gg)$, the hyperalgebra
$\operatorname{Dist}(G)$ \cite[I.7]{J} and its commutative analogue the algebra of divided powers
$D(\gg)$, and the restricted enveloping algebra $u(\gg)$ which is a
quotient of $U(\gg)$ and a subalgebra of $\operatorname{Dist}(G)$.  Many homological
questions about these algebras are still unanswered: for example, while the
$G$-invariants, and $\gg$-invariants of $U(\gg)$ and $S(\gg)$ are known
\cite{veldkamp}, the higher Hochschild cohomology groups are not. Even
the centres of the restricted enveloping algebra and the hyperalgebra
are not known in general.

First Hochschild cohomology of $U(\gg)$, which equals the space of outer
derivations of $U(\gg)$, is of particular interest. The work of Riche
\cite{Riche} shows that the restricted
enveloping algebras are Koszul,
which implies that there is a $\mathbb{Z}_{\geq 0}$ grading on these
algebras. This leads to a derivation multiplying the degree $i$
component by $i$: does this extend to a derivation of
$\operatorname{Dist}(G)$, or lift
to a derivation of $U(\gg)$?  

The aim of this paper is to compute the Hochschild cohomology in the
simplest case, when $\gg=\ssl_2(k)$.  In Section 5 we give a 
description of the Hochschild cohomology groups, including the
dimensions of their graded pieces and their structure as a module over the
zeroth Hochschild cohomology.  In section 6 we study invariants on the
divided power algebra $D(\ssl_2(k))$, which turns out to be related to
the top Hochschild cohomology group. We compute the dimensions of the
graded pieces of the $\gg$-invariants and show
that the $G$-invariants have the same Hilbert series as $S(\gg)^G$.

From now on we let $k$ be a field, $\gg = \ssl_2(k)$ and $U=U(\gg)$ be
its universal enveloping algebra.  The following method of describing
the Hochschild (co)homology groups of $U$ is well-known, e.g.
\cite[Proposition 5]{kassel}, \cite[XIII]{CE}.
The universal enveloping algebra $U$ is a Hopf algebra with antipode
$\lambda(x)=-x$ and comultiplication $\Delta(x)=x\otimes 1 + 1 \otimes x$
for $x \in \gg$. Let $U^e = U \otimes_k
U^{\textrm{op}}$ be the enveloping algebra of $U$, so there is an
algebra homomorphism $(1\otimes \lambda) \circ \Delta : U \to U^e$ making
$U^e$ into a free $U$-module (on the generators $b \otimes 1$ for
$b$ in a PBW basis of $U$).  The induced module $k \mathord\uparrow _U
^{U^e}$ is isomorphic to $U$ and the restricted module $U|^{U^e}_U$ is
the adjoint module $U^{\textrm{ad}}$, so the Eckmann-Shapiro Lemma gives
\begin{equation}\label{ES}
\HH^*(U) = \ext^*_{U^e}(U,U) \cong \ext^*_{U^e}(k \mathord\uparrow
^{U^e}_U,U) \cong \ext^*_U(k, U^{\textrm{ad}}).
\end{equation}

If $k$ has characteristic zero, the structure of $\HH^*(U)$ follows
immediately: in that case the Whitehead lemma \cite[7.8.9]{weibel} says that $\ext^*_U(k,L)=0$ whenever $L$ is a
nontrivial simple $U$-module,
so \[\HH^*(U) \cong \ext^*_U(k,k)
	\otimes_k Z(U)=\ext^*_U(k,k)\otimes_k k[\tilde c]\]  where
$\tilde c$ is the
Casimir element of $U$ and
$\ext^*_U(k,k)$ is an exterior algebra with one generator of degree 3
arising from the Killing form.

To calculate the Hochschild cohomology in the case when the characteristic
of $k$ is greater than two we will use the isomorphisms
\begin{equation*}
	\HH^*(U) \cong \ext^*_U(k,U^{\textrm{ad}}) \cong \ext^*_U(k, S)
\end{equation*}
where $S$ is the symmetric algebra on $\gg$, the first isomorphism is
from (\ref{ES}) and the second is from the isomorphism of $U$-modules
$U^{\textrm{ad}}\cong S$ of \cite{FP}.
%This result shows
%$\HH^*(U)$ agrees with the Poisson cohomology of $S$ equipped with the
%Poisson bracket induced by the Lie bracket of $\gg$.

\subsection{Notation}
Let $k$ be a field of characteristic $p>2$ and $$e= \begin{pmatrix}
    0&1\\0&0
\end{pmatrix} ,h= \begin{pmatrix}
    1&0\\0&-1
\end{pmatrix}, f= \begin{pmatrix}
    0&0\\1&0
\end{pmatrix}$$ be the usual
basis of $\gg=\ssl_2(k)$.  Let $S^n$ be the $n$th symmetric power of
the adjoint $\gg$-module $\gg$ and $S = \bigoplus_{n\geq 0}S^n$ the symmetric
algebra on the adjoint $\gg$-module. Let $c = h^2 + 4ef \in S^\gg$, let
$Z$ be the subalgebra of $S$ generated by $c, e^p, h^p, f^p$ so that
$Z=S^\gg$ and let $Z_0$ be the subalgebra of $Z$ generated by $e^p, f^p$
and $h^p$.  
%Multiplication by an element of $Z$ is a $\gg$-endomorphism
%of $S$, so each Hochschild cohomology group is a $Z$-module.
%% put it later, when it's relevant!

Let $L(r)$ be the simple $\gg$-module with dimension
$r+1$.  The natural module $L(1)$ has a basis $x,y$ with $h\cdot x = x,
h\cdot y = -y$, and its $n$th symmetric power is the dual Weyl module
$\nabla(n)$. We define $\nabla= \bigoplus_{n\geq 0}\nabla(2n)$, a
subalgebra of $k[x,y]$.  The modules $S^n$ and $\nabla(n)$ all admit
\emph{weight gradings}: $\mathbb{Z}$-gradings such that $e,h,f$ have
degrees $2,0,-2$ and $x,y$ have degrees $1,-1$ respectively, and such
that $h$ acts by multiplication by $\lambda$ on the homogeneous
component of degree $\lambda$.

Let $G= \mathsf{SL}_2(k)$ and $G_1$ be its first Frobenius kernel; the
modules $S^n$, $L(r)$ and $\nabla(n)$ are also $G$ and $G_1$-modules.
If $M$ is a $G$-module we write $M^F$ for its first Frobenius twist.

\subsection{The standard resolution}\label{stdres}
There is a resolution of the trivial $U$-module
\begin{equation*}
	0 \to U\otimes_k \wedge^3  \gg \stackrel {\delta_3}\to U\otimes_k
	\wedge ^2 \gg \stackrel {\delta_2} \to U
	\otimes_k \gg \stackrel{\delta_1} \to U \to  k \to 0
\end{equation*}
in which the differentials are given by
\begin{multline*}
	\delta_n (1\otimes x_1 \wedge \cdots \wedge x_n)=\sum_i
	(-1)^{i+1} x_i \otimes x_1\wedge \cdots \wedge \widehat{x_i}
	\wedge \cdots \wedge x_n  \\ + \sum_{i<j} (-1)^{i+j}\otimes
	[x_i,x_j] \wedge x_1 \wedge \cdots \wedge \widehat{x_i} \wedge
	\cdots \wedge \widehat{x_j} \wedge \cdots \wedge x_n
\end{multline*}
for $n \geq 1$ and $\delta_1(1\otimes x)=x$ for $x, x_i \in \gg$. Any
cocycles in this paper are defined on this resolution.

Module homomorphisms from the degree $i$ term of this resolution to a
$\gg$-module $M$ correspond to linear maps $\wedge^i\gg \to M$. Under
this correspondence, 1-cocycles are linear maps $\alpha : \gg \to M$
such that
\begin{equation}\begin{aligned}\label{1cocycle}
	e\cdot \alpha(f)-f\cdot \alpha(e) &= \alpha(h) \\
(h-2)\cdot \alpha(e)&=e\cdot \alpha(h) \\
(h+2)\cdot \alpha(f) &=f\cdot \alpha(h)
\end{aligned}\end{equation}
and 1-coboundaries are maps $\gg \to M$ of the form
\begin{equation}\label{1coboundary}
	\delta_m(r) = r \cdot m
\end{equation}
for $m\in M$.
2-cocycles are maps $\alpha : \wedge^2\gg \to M$ such that
\begin{equation}\label{2cocycle}
	e\cdot \alpha(h\wedge f)+f\cdot\alpha(e\wedge h) = h\cdot
\alpha(e\wedge f),
\end{equation}
and 2-coboundaries are spanned by maps of the form $r_m, s_m, t_m$ for
$m\in M$, where
\begin{equation}\label{2bdry}
\begin{array}{lll}
r_m(e\wedge h)=(2-h)\cdot m& r_m(e\wedge f)=-f\cdot m&r_m(h\wedge f)=0  \\
 s_m(e\wedge h)=0 & s_m(e\wedge f)=e\cdot m& s_m(h \wedge f)=(2+h)\cdot m\\
t_m(e\wedge h)=e\cdot m& t_m(e\wedge f)=-m&t_m(h\wedge f) = -f\cdot m .
\end{array}
\end{equation}
When we identify $\hom_U(U\otimes \wedge^3 \gg,M)$ with $M$, the space
of boundaries becomes $\gg\cdot M$ and so 
$\ext^3_U(k,M)$ can be identified with $M/\gg\cdot M \cong \hom_U(M,k)^*$.

\section{The long exact sequence}\label{les_sec}
In this section we derive a long exact sequence that will allow us to
compute the Hochschild cohomology groups $\ext^m_U(k,S)$.

The algebra map $\phi: S \to \nabla$ determined by $\phi(e)=-x^2/2,
\phi(h)=xy, \phi(f)=y^2/2$ is a homomorphism of $G$- (and hence $\gg$-)
modules, and $\ker\phi$ is the ideal generated by $c$ so we get
for each $n$ a short exact sequence of $G$-modules
\begin{equation}\label{SSnabla}
	0 \to S^{n-2} \stackrel{c}\to S^n \stackrel{\phi}\to \nabla(2n) \to 0.
\end{equation}
This gives rise to an action of $Z_0$ on $\nabla$, hence on the
spaces of cocycles and coboundaries for the standard resolution with
values in $\nabla$ and 
the cohomology groups $\ext^*_U(k,\nabla)$. The generators $e^p$,
$h^p$ and $f^p$ of $Z_0$ act on $\nabla$ via multiplication by $-x^{2p}/2, x^py^p$,
and $y^{2p}/2$ respectively.

Applying $\hom_U(k,-)$ to (\ref{SSnabla}) leads to the long exact sequence
\begin{equation}\label{les}
	\cdots \to \ext^i_U(k,S^{n-2}) \stackrel{c}\to \ext^i_U(k,S^n)
\stackrel{\phi}\to \ext^i_U(k,\nabla(2n)) \stackrel{\omega_i}\to
\ext^{i+1}_U(k,S^{n-2}) \to \cdots
\end{equation}
where $\omega_i$ is the $i$th connecting homomorphism. Our strategy for
computing $\ext^*_U(k,S)$ is to determine the cohomology of the
dual Weyl modules and its $Z_0$-module structure, then to use this long
exact sequence. In particular we will show that the connecting
homomorphisms are almost always zero.

\section{Cohomology of dual Weyl modules} \label{dualweyl}
Let $2n=qp+r$ with $0\leq r<p$. The dual of
(1.5)(1) in \cite{K} is a
non-split exact sequence of $G$-modules
\begin{equation}\label{nabla_structure}
	0 \to \nabla(r) \otimes \nabla(q)^F \to \nabla (2n) \to
\nabla(p-2-r)
\otimes \nabla(q-1)^F \to 0.
\end{equation}
%The maps are $x^ay^{r-a}\otimes X^b Y^{q-b} \mapsto
%x^{a+pb}y^{r-a+p(q-b)}$ and $x^{\alpha+p\beta}y^{(q-\beta)p+r-\alpha}
%\mapsto 0$ if $\alpha \leq r$ and $x^{p-1-a}y^{r'-p+1+a}\otimes X^\beta
%Y^{q-1-\beta}$ otherwise.
As $\mathfrak{g}$-modules, the outer terms of the short exact
sequence are isomorphic to $L(r)^{\oplus (q+1)}$ and $L(p-2-r)^{\oplus q}$
respectively. We will use the long exact sequence obtained by applying
$\hom_U(k,-)$ to (\ref{nabla_structure}) to compute
$\ext_U^*(k,\nabla(2n))$, so we need the cohomology of the simple
$\gg$-modules.

\begin{prop}\label{cohom_simples}
Let $0 \leq r < p$. As $G$-modules 
\begin{equation*}\ext^m_U(k,L(r)) \cong \begin{cases}
k & r=0, m=0, 3\\
\nabla(1)^F& r=p-2, m=1,2\\
0 & \text{otherwise.}
  \end{cases} 
\end{equation*}
\end{prop}

This is a straightforward computation using the standard resolution. We
can now find the $U$-cohomology of the dual Weyl modules. The 
cohomology groups $\ext^1_{G_1}(k,\nabla(2n))$ for the first Frobenius
kernel, or equivalently, for $\gg$ as a restricted Lie algebra,
were determined in \cite{K}.
%or can be found using the same methods? check if K does m>1
They are the subspaces of our
groups corresponding to restricted extensions.

\begin{thm}\label{extnabla}
Suppose $qp+r$ is even and $0\leq r<p$. Then
$\ext^m_U(k, \nabla (qp+r))$ is zero unless
$r=0$ or $p-2$, and as $G$-modules
\begin{align*}
    \ext_U^m(k, \nabla (qp)) &\cong\begin{cases}
		\nabla(q)^F & m=0 \\
\nabla(q-1)^F\otimes \nabla(1)^F  & m=1, q\geq 2 \\
\nabla(q-2)^F & m=2, q \geq 2 \\
k & m=3, q=0 \\
0 & \text{otherwise.}
	\end{cases} \\
    \ext_U^m(k,\nabla (qp+p-2)) &\cong \begin{cases}
		\nabla(q+1)^F & m=1 \\
\nabla(q)^F\otimes\nabla(1)^F & m=2 \\
		\nabla(q-1)^F & m=3 \\
0 & \text{otherwise}.
	\end{cases}
\end{align*}
\end{thm}
%
%By de Visscher \emph{Extensions of modules for $SL(2,K)$}, we have 
%\begin{equation*}
%	0 \to \nabla(z-1)^F \to \nabla(z)^F \otimes \nabla(1)^F \to
%\nabla(z+1)^F \to 0
%\end{equation*}
%is split iff $z \not\equiv -1 \mod p$, so the tensor products can't
%always be decomposed.

\begin{proof}
The long exact sequence obtained by applying
$\hom_U(k,-)$ to (\ref{nabla_structure}) is
\begin{multline*}
\cdots\to \ext^m_U(k, \nabla (r) \otimes \nabla (q)^F)
\to \ext^m_U (k,\nabla (qp+r)) 
\to \ext^m_U(k, \nabla (p-2-r) \otimes \nabla(q-1)^F)  \\
\to \ext^{m+1}_U(k, \nabla (r) \otimes \nabla (q)^F) 
\to \cdots
%\ext^2_U(k, \nabla (2n)) 
%\to \ext^2_U(k,\nabla(p-2-r)  \otimes \nabla(q-1)^F) \\
%\to \ext^3_U( k, \nabla (r) \otimes \nabla (q)^F)
%\to \ext^3_U(k, \nabla (2n))
%\to \ext^3_U(k,\nabla (p-2-r) \otimes \nabla(q-1)^F)
%\to 0.
\end{multline*}

This and Proposition \ref{cohom_simples} imply that the
Ext groups vanish for $r \neq 0,p-2$ and the statement is correct for $q=0$.  The result 
for $m=0$ follows from (\ref{nabla_structure}).

\begin{itemize}
\item Suppose $r=0,m=1$. The first $\ext^1$ and $\ext^2$ groups appearing in the long
exact sequence are zero by Proposition \ref{cohom_simples}, so
\begin{align*}
	\ext^1_U(k,\nabla(2n))  &\cong \ext^1_U(k, \nabla(p-2)\otimes
\nabla(q-1)^F ) \\
&\cong \ext^1_U(k,L(p-2))\otimes \nabla(q-1)^F
\end{align*}
since $U$ acts trivially on $\nabla(q-1)^F$. By Proposition
\ref{cohom_simples}, $\ext^1_U(k,\nabla(qp)) \cong \nabla(1)^F
\otimes \nabla (q-1)^F$.
\item Suppose $r=0,m=3$.
Recall from \S \ref{stdres} that
$\ext^3_U(k,\nabla(qp))\cong \nabla(qp)/\mathfrak{g}\cdot
\nabla(qp)$. The action of $\gg$ on $\nabla$ is \begin{align*}
h\cdot x^ay^b& =(a-b)x^ay^b\\ 
e\cdot x^{a-1}y^{b+1}&=(b+1)x^ay^b\\ 
f\cdot x^{a+1}y^{b-1} &=(a+1)x^ay^b
\end{align*}
so $x^ay^b$ is not in the image of the action of $\mathfrak{g}$ if and
only if $a\equiv b \equiv -1 \mod p$ or $a=b=0$. The claim about
$\ext^3_U(k,\nabla(qp))$ follows.

\item Suppose $r=0,m=2$.
For $q>0$ we have $\ext^2_U(k,\nabla(0)\otimes\nabla(q)^F)=0$ by
Proposition \ref{cohom_simples} as
$\nabla(0)\otimes \nabla(q)^F$ is trivial for $U$. Also
\begin{align*}
\ext^2_U(k, \nabla(p-2) \otimes \nabla(q-1)^F) &\cong
                                                 \ext^2_U(k,L(p-2))\otimes
                                                 \nabla(q-1)^F \\
&\cong \nabla(1)^F \otimes \nabla(q-1)^F,
\end{align*}
$\ext^3_U(k,\nabla(0)\otimes \nabla(q)^F) \cong
\ext^3_U(k,k)\otimes\nabla(q)^F \cong \nabla(q)^F$, and
 $\ext^3_U(k,\nabla(qp))=0$ by the previous case.  Therefore 
the long exact sequence contains
\begin{equation*}
0 \to \ext^2_U(k, \nabla (qp)) \to \nabla(1)^F \otimes
\nabla(q-1)^F \to \nabla (q)^F \to 0.
\end{equation*}
In \cite[p.416, above (6)]{deV}, De Visscher shows
$\hom_G(\nabla(q-1),\nabla(1)\otimes\nabla(q))\cong k$ so that
$\hom_G(\nabla(q-1)\otimes\nabla(1),\nabla(q))\cong k$, and (6) itself
shows that
there is a homomorphism $\nabla(1)\otimes \nabla(q-1)
\to \nabla(q)$ with kernel $\nabla(q-2)$. Because
the space of homomorphisms is one-dimensional we must have
$\ext^2_U(k,\nabla(qp))\cong \nabla(q-2)^F$.
\item Suppose $r=p-2,m=1$. Then $\hom_U(k,\nabla(qp+p-2))=0$ and
the subsequent part of the long exact sequence is 
\begin{equation*}
0 \to \nabla(q-1)^F \to 
\nabla(1)^F\otimes\nabla (q)^F  \to 
\ext^1_U(k,\nabla (qp+p-2)) \to
0
\end{equation*}
%(the tensor product has a filtration with terms $\nabla(q-1)^F$ and
%$\nabla(q+1)^F$ but it isn't split if $q\equiv -1 \mod p$)
and so by \cite{deV} again this ext group is isomorphic to $\nabla (q+1)^F$.

\item Suppose $r=p-2, m=2$.  Part of the long exact
sequence is
\begin{multline*}
\ext^1_U(k,\nabla(0)\otimes\nabla(q-1)^F)\to
\ext^2_U(k,\nabla(p-2)\otimes\nabla(q)^F) \to
\ext^2_U(k,\nabla(qp+p-2)) \\ \to
\ext^2_U(k,\nabla(0)\otimes\nabla(p-2)^F)
\end{multline*}
The first and last terms are zero, as $\nabla(0)\otimes\nabla(q-1)^F$
is trivial as a $U$-module.  The second is isomorphic to
$\ext^2_U(k,\nabla(p-2))\otimes \nabla(q)^F$ as $\nabla(q)^F$ is
trivial as a $U$-module, so to $\nabla(1)^F\otimes \nabla(q)^F$ by
Proposition \ref{cohom_simples}.  It follows that $ \ext^2_U(k,\nabla
(qp+p-2))\cong \nabla(1)^F \otimes \nabla (q)^F $.

\item Suppose $r=p-2,m=3$.
We have 
\begin{align*}\ext^3_U(k,\nabla(0)\otimes\nabla(q-1)^F) &\cong
\ext^3_U(k,k)\otimes\nabla(q-1)^F\cong \nabla(q-1)^F\\ 
\ext^3_U(k,\nabla(p-2)\otimes\nabla(q)^F) &\cong
\ext^3_U(k,\nabla(p-2))\otimes \nabla(q)^F = 0,
\end{align*}
so the final two
nonzero terms of the long exact sequence give
$\ext^3_U(k,\nabla (qp+p-2)) \cong \nabla(q-1)^F$. \qedhere
\end{itemize}
\end{proof}

\subsection{$Z_0$-action on $\ext^1_U(k,\nabla)$}

\begin{prop}\label{ext1nabla}
	$\ext^1_U(k,\nabla)$ is generated as a $Z_0$-module by
$\ext^1_U(k, \nabla(2p-2))$ and $\ext^1_U(k, \nabla(2p))$.
\end{prop}
\begin{proof}
Since $\ext^1_U(k,\nabla(2n))$ is only nonzero for $n\equiv 0$ or $p-1
\mod p$, this amounts to saying that for every odd $q$ we have
$\ext^1_U(k,\nabla(qp+p-2)) \subseteq Z_0 \ext^1_U(k,\nabla(p-2))$ and
for every even $q$ we have
$\ext^1_U(k,\nabla(qp)) \subseteq Z_0 \ext^1_U(k,\nabla(2p))$.
To show this we find cocycles for the standard
resolution whose classes form bases of these ext groups and which are
equal to
powers of $x^{2p},y^{2p}$, and $x^py^p$ times cocycles with values in
$\nabla(2p-2)$ and $\nabla(2p)$. This is enough since
the action of the generators of
$Z_0$ on these spaces of cocycles is via multiplication by $x^{2p}/2,
y^{2p}/2$, and $x^py^p$. 

We record linear maps $\alpha$ on
$\gg$ as triples $(\alpha(e), \alpha(h), \alpha(f))$. A weight homogeneous
linear map $\alpha: \gg \to \nabla(2n)$ has the form
\begin{equation*}
	(\lambda_e x^{i+1}y^{2n-i-1}, \lambda_h x^i y^{2n-i}, \lambda_f
x^{i-1}y^{2n-i+1})
\end{equation*}
for some $\lambda_e,\lambda_h,\lambda_f\in k$.  The cocycle condition
(\ref{1cocycle}) translates to
\begin{align*}
	(n+1)(\lambda_f-\lambda_e) &= \lambda _h \\
n \lambda_h &= 0
\end{align*}
and the coboundary $\delta_{x^iy^{2n-i}}$ is
$((2n-i)x^{i+1}y^{2n-i-1}, 2(i-n)x^iy^{2n-i},ix^{i-1}y^{2n-i})$.

%If the weight degree $j=2(i-n)$ of $\alpha$ is not divisible by $p$
%then the cocycle
%conditions imply that $\alpha = \delta_{\alpha(h)/j}$, so if $\alpha$
%represents a nonzero cohomology
%element we must have $i \equiv n \mod p$.

Firstly let $2n=qp+p-2$, so $\ext^1_U(k, \nabla(2n))\cong \nabla(q+1)^F$.
The maps
\begin{equation*} %\label{nabla2pm2gens}
    \begin{array}{lllll}
        d_e=(y^{2p-2},0,0),& & d=(x^py^{p-2},0,-x^{p-2}y^p), & & d_f=(0,0,x^{2p-2})
    \end{array}
\end{equation*}
are cocycles for $\ext^1_U(k,\nabla(2p-2))$. 
%are not boundaries (as any boundary has
%$\lambda_e=\lambda_f$), and have linearly independent cohomology
%classes since their weight degrees $-2p, 0,2p$ are distinct. Their
%classes therefore span $\ext^1_U(k,\nabla(2q-2))$.

For $q>1$ odd the $q+2$ cocycles
\[ \begin{array}{lllll} y^{(q-1)p}d_e, & & x^{2sp}y^{(q-2s-1)p}d, & &
x^{2sp}y^{(q-2s-1)p}d_f\end{array}\]
for $0\leq s\leq (q-1)/2$ are not boundaries (as any boundary has
$\lambda_e=\lambda_f$), and represent linearly independent elements of
$\ext^1_U(k,\nabla(qp+p-2))$ since
their weight degrees are distinct. They therefore form a basis of
$\ext^1_U(k,\nabla(qp+p-2))$ and are visibly in the image of the $Z_0$
action.

Now let $r=0$, so that $\ext^1_U(k,\nabla(2n)) \cong \nabla(q-1)^F
\otimes \nabla(1)^F$ has dimension $2q$. Consider the cocycles
\begin{equation*}
    \begin{array}{lll}
        \tilde E = (0,x^{2p},x^{2p-1}y), & & \tilde
H=(x^{p+1}y^{p-1},0,x^{p-1}y^{p+1}), \\ \tilde F = (xy^{2p-1},-y^{2p},0),
&& \tilde C = (x^{p+1}y^{p-1}, -2x^py^p, -x^{p-1}y^{p+1})
\end{array}
\end{equation*}
with values in $\nabla(2p)$. 
%form a basis of $\ext^1_U(k,\nabla(2p))$ as when $n\equiv 0 \mod p$
%there are no nonzero boundaries with weight degree divisible by $p$.
For $q>2$ even the $2q$ cocycles
\begin{equation*}
    \begin{array}{lllllll}
        x^{2sp}y^{(q-2s-2)p}E, && x^{2sp}y^{(q-2s-2)p}H,&&
        x^{2sp}y^{(q-2s-2)p}F,&& x^{2sp}y^{(q-2s-2)p}C
    \end{array}
\end{equation*}
for $0\leq s <q/2$ have weight degrees
$(4(s+1)-q)p$, $(4s+2-q)p$, $(4s-q)p$, and $(4s+2-q)p$
respectively. They are not boundaries, as there are no nonzero
boundaries with weight degree divisible by $p$ with values in
$\nabla(qp)$. They are linearly independent as the only pairs amongst
them with equal weight degree are
$x^{2sp}y^{(q-2s-2)p}E$ and $x^{2(s+1)p}y^{(q-2s-4)p}F$, and
$x^{2sp}y^{(q-2s-2)p}H$ and $x^{2sp}y^{(q-2s-2)p}C$, but 
multiples of $E$ kill $e$ while multiples of $F$ do not, and multiples
of $H$ kill $h$ while multiples of $C$ do not.
Therefore their cohomology classes span $\ext^1_U(k,\nabla(qp))$ and
are clearly in the image of the $Z_0$-action on $\ext^1_U(k,\nabla(2p))$.
%
%
%\begin{equation*}
%	(x^{sp+1}y^{(q-s)p-1}, -x^{sp}y^{(q-s)p},0) \text{ and } (0,
%x^{sp}y^{(q-s)p}, x^{sp-1}y^{(q-s)p+1})
%\end{equation*}
%for $0\leq s < q$ in the first case and $0<s \leq q$ in the second are
%cocycles, are not boundaries since there are no nonzero coboundaries of
%these weight degrees, and are linearly independent since they have
%different weight degrees. They therefore span the ext group. Since
%\begin{equation}\label{ext_1_2p_generators}
%	(x^{sp+1}y^{(q-s)p-1}, -x^{sp}y^{(q-s)p},0) = \begin{cases}
%		(x^{2p})^{s/2}(y^{2p})^{(q-s-2)/2}(xy^{2p-1}, -y^{2p},
%0) & s \text{ even} \\
%(x^{2p})^{(s-1)/2} (y^{2p})^{(q-s-1)/2}(x^{p+1}y^{p-1}, -x^py^p,0) & s
%\text{ odd}
%	\end{cases}
%\end{equation}
%with a similar result for the second cocycle, it follows that
%$\ext^1_U(k,\nabla(qp)) \subseteq Z_0\ext^1_U(k, \nabla(2p))$.
\end{proof}

\subsection{$Z_0$-action on $\ext^2_U(k,\nabla)$}
\begin{prop} \label{ext2nabla}
	$\ext^2_U(k, \nabla)$ is generated as a $Z_0$-module by
$\ext^2_U(k, \nabla(2p-2))$ and $\ext^2_U(k, \nabla(2p))$.
\end{prop}

\begin{proof}
We use the same method as the proof of Proposition
\ref{ext1nabla}, and we record maps $\alpha : \wedge^2\gg \to \nabla(2n)$ as
triples \[(\alpha(e\wedge h), \alpha(e\wedge f),\alpha(h\wedge f)).\]
If such a map is weight homogeneous, it has the form
\[ (\lambda e x^{i+1}y^{2n-i-1}, \lambda h x^iy^{2n-i}, \lambda_f
  x^{i-1}y^{2n-i+1}),\]
  and from (\ref{2cocycle}) such a map is a cocycle if and only if 
\[ (i+1)\lambda_e+(2n-i+1)\lambda_f=2(i-n)\lambda_h.\]

Consider the cocycle $d=(x^{p+1}y^{p-1},0,x^{p-1}y^{p+1})$. For $q$
even, this gives
rise to $q-1$ cocycles with values in $\nabla(qp)$
\begin{equation*}
    \begin{array}{ll} x^{2sp}y^{(q-2s-2)p}d & 0\leq s<\frac{q}{2} \\
        x^py^p \cdot x^{2sp}y^{(q-2s-4)p} d & 0\leq s < \frac{q-2}{2},
    \end{array}
\end{equation*}
which are in the image of the $Z_0$ action. These cannot be
coboundaries, as all of the boundaries (\ref{2bdry}) with appropriate
weight degree vanish on $e\wedge h$ and $h\wedge f$.
% r_m: must be m=x^{mp+1}y^{(q-m)p-1}, so (2-h)m=0
% s_m: similar
% t_m: must be m=x^{mp}y^{(q-m)p}, so f\cdot m = e\cdot m=0
Their cohomology classes are linearly independent since their weight
degrees are distinct, so they span $\ext^2_U(k,\nabla(qp)) \cong
\nabla(q-2)^F$.

Now consider the cocycles
\begin{equation*}
    \begin{array}{lll}
        \tilde R_e=(y^{2p-2},0,0),&& \tilde R_f=(0,0,x^{2p-2}), \\
        A=(x^py^{p-2},0,0), && B=(0,0,x^{p-2}y^p)
    \end{array}
\end{equation*}
which, when multiplied by $x^{2sp}y^{(q-2s-1)p}$ for $0\leq s \leq
(q-1)/2$ give rise to $2q+2$ cocycles with values in
$\nabla(qp+p-2)$.  Coboundaries with the appropriate weight degrees
again vanish on $e\wedge h$ and $h\wedge f$, and the only pairs
with the same weight degrees are $x^{2sp}y^{(q-2s-1)p}A$ and
$x^{2sp}y^{(q-2s-1)p}B$, and $x^{2sp}y^{(q-2s-1)p}\tilde R_e$ and
$x^{2(s-1)p}y^{(q-2s+1)p}\tilde R_f$ which are linearly independent
since in each case one vanishes on $e\wedge h$ and the other does
not. It follows that the classes of these cocycles, which are in the image of the $Z_0$
action on $\ext^2_U(k,\nabla(2p-2))$, span $\ext^2_U(k,\nabla(qp+p-2))\cong
\nabla(q)^F\otimes \nabla(1)^F$.
\end{proof}

\section{Connecting homomorphisms} \label{conhom}
\begin{prop}\label{w0w1}
	In the long exact sequence (\ref{les}), $\omega_0=\omega_1=0$.
\end{prop}
\begin{proof}
$\nabla^\gg$ is spanned by monomials of the form $x^{ap}y^{bp}$,
and these have 
preimages in $S^\gg$ of the form $e^{rp}h^{sp}y^{tp}$. Therefore
$\phi: S^\gg \to \nabla^\gg$ is onto and $\omega_0=0$.

To show $\omega_1=0$, we prove $\phi:\ext^1_U(k,S) \to
\ext^1_U(k,\nabla)$ is onto.  Because $\phi$ is compatible with the
$Z_0$-action it is enough to find preimages for the $Z_0$-generating set
for $\ext^1_U(k,\nabla)$ provided by Proposition \ref{ext1nabla}.
Consider the following cocycles $\gg \to S$:
\begin{equation}\label{genset1}
\begin{array}{lll}
    \delta_e = (f^{p-1},0,0) & &
\delta_f = (0,0,e^{p-1})
\end{array}\end{equation}
\begin{equation*}
 \delta = \left( \frac{h^p-hc^{(p-1)/2}}{4f} 
 ,c^{(p-1)/2}, \frac{h^p-hc^{(p-1)/2}}{4e} \right) \end{equation*}
\begin{equation*}
    \begin{array}{lll}
E = (0,2e^p,-he^{p-1}) & H=(-2eh^{p-1},0,2fh^{p-1}) &
F=(hf^{p-1},-2f^p,0) \end{array}\end{equation*}
\begin{equation*}
  C=\left(\frac{c^{(p+1)/2}-h^{p+1}}{4f},h^p, \frac{c^{(p+1)/2}-h^{p+1}}{4e}
\right).
\end{equation*}
In the notation used in the proof of Proposition \ref{ext1nabla},
$\phi(\delta_e)=d_e$, $\phi(\delta)=d$, $\phi(\delta_f)=d_f$, 
$\phi(E)=\tilde E$, $\phi(H)=\tilde H$, $\phi(F)=\tilde F$, and $\phi
(C)=\tilde C$.
\end{proof}

To deal with $\omega_2$ we need the following:
\begin{lem}
If $n\neq p-1$ and $cy \in \gg\cdot S^{n}$ then $y \in \gg\cdot S^{n-2}$.
\end{lem}
\begin{proof}
    Suppose that $cy \in \gg \cdot S^n$. We may assume
$y$ is weight homogeneous, and that its weight
degree is divisible by $p$ since otherwise it lies in the image of the
action of $h$.  As vector spaces, $S^n = cS^{n-2}\oplus K$ where $K$ is
the subspace spanned by all monomials $e^if^{n-i}$ and $e^if^{n-1-i}h$,
so we can write
\begin{equation*}
	cy =  e\cdot cy_e +  f\cdot cy_f + e\cdot d_e + f \cdot d_f
\end{equation*}
where $y_e,y_f \in S^{n-2}$ and $d_e,d_f \in K$. Replacing $y$
by $y-e\cdot y_e-f\cdot y_f$ it is enough to prove that if
\begin{equation*}
	cy = e\cdot d_e+f\cdot d_f
\end{equation*}
for $d_e,d_f\in K$ then $y\in \gg\cdot S^{n-2}$. As $y$ is
weight homogeneous of degree divisible by $p$,
\begin{align*}
    d_e &\in \operatorname{span} \{ e^{ap}
        f^{bp}e^jf^{j+1},e^{ap}f^{bp}e^{l}f^{l+1}h : a,b \geq 0,
    2j\equiv 2l+1\equiv n-1 \mod p\} \\
    d_f & \in
    \operatorname{span}\{e^{ap}f^{bp}e^{j+1}f^j,e^{ap}f^{bp}e^{l+1}f^{l}
    h : a,b\geq 0, 2j\equiv 2l+1\equiv n-1 \mod p \}
\end{align*}
Since $e \cdot e^jf^{j+1}=-f\cdot e^{j+1}f^j$ and $e\cdot
e^{l}f^{l+1}h = -f\cdot e^{l+1}f^{l}h$ we may assume $d_f=0$.
Now
\begin{align*}
	e \cdot e^jf^{j+1} &= (j+1) e^jf^jh \in K \\
e \cdot e^{l}f^{l+1}h &= (l+1)ce^{l}f^{l} -(4l+6)e^{l+1}f^{l+1}
\end{align*}
so for $e\cdot d_e \in cS^{n-2}$ it must be that $d_e$ is a multiple of $e^{ap}f^{bp}e^{l}f^{l+1}h$, and
then only when $4l+6\equiv 0 \mod p$, that is, when $n\equiv p -1 \mod
p$.  To finish the proof it is enough to show that
$e^{ap}f^{bp}e^{l}f^{l+1}h\in \gg \cdot S^{n-2}$ if $a$ or $b\neq
0$: the case $a\neq 0$ follows because
\begin{equation*}
	e^{p+(p-3)/2}f^{(p-3)/2} = e \cdot \frac{-1}{2}
\sum_{i=0}^{(p-3)/2} (-4)^{-i} h^{2i+1}e^{p + (p-5)/2 -i} f^{(p-3)/2-i}
\end{equation*}
and the $b\neq 0$ case is similar.
\end{proof}

\begin{cor} \label{w2}
    $\omega_2: \ext^2_U(k,\nabla(2n)) \to \ext^3_U(k,S^{n-2})$ is zero
    unless $n= p-1$ when it has rank one.
\end{cor}
\begin{proof}
    Because $\ext^3_U(k,S)$ can be identified with $S/ \gg \cdot S$, the
    previous lemma shows that $c:\ext^3_U(k,S^{n-2}) \to
\ext^3_U(k, S^n)$ is injective, hence $\omega_2$
is zero, except when $n=p-1$. For $m<p$, $S^m$ is self-dual
so for $m$ even
$\ext_U^3(k,S^{m}) \cong \hom_U(S^m,k)^* \cong \hom_U(k,(S^m)^*)^*
\cong \hom_U(k, S^{m})^* \cong k$.
Thus the part of the long exact sequence (\ref{les}) starting at
$\ext^2_U(k,\nabla(2(p-1))$ is
\begin{equation*}
    \nabla(1)^F\otimes\nabla(1)^F \stackrel{\omega_2}{\to} k \to k \to
\nabla(0)^F \to 0
\end{equation*}
exactness of which implies that $\omega_2$ has rank one when $n=p-1$.
\end{proof}

% h^2 \cong -4ef \mod c, so C is congr to (h^{p-1}e, h^p, h^{p-1}f) mod
% c
% and \delta is congr to (-h^{p-2}e, 0, -h^{p-2}f) mod c
%
% \phi \delta e = (y^{2p-2} 0 0)
% \phi \delta f = (0 0 x^{2p-2})
% \phi \delta = (1/2)(x^py^{p-2}, 0, -x^{p-2}y^p)
%
% \phi E = (0 -x^{2p}, -x^{2p-1}y)
% \phi F = (xy^{2p-1}, -y^{2p},0)
% \phi H = (x^{p+1}y^{p-1}, 0, x^{p-1}y^{p+1})
% \phi C = (1/2)(-x^{p+1}y^{p-1}, 2x^py^p, x^{p-1}y^{p+1})
% so \phi 2C = (-x^{p+1}y^{p-1} 2x^py^p x^{p-1}y^{p+1})
% \phi H - \phi 2C = (2x^{p+1}y^{p-1}, -2x^py^p, 0)
\section{Hochschild cohomology}
In this section we find the dimensions of the graded pieces of the
Hochschild cohomology groups $\ext^*_U(k,S)$ and generators and
relations for these groups as $Z$-modules. 

\subsection{$\HH^1(U)$}
The connecting homomorphisms $\omega_0$ and $\omega_1$ in (\ref{les})
are zero by Proposition \ref{w0w1} so for each $n$ there is a short
exact sequence
\begin{equation}\label{ses1}
0 \to \ext^1_U(k, S^{n-2})\to \ext^1_U(k,S^n) \to \ext_U^1(k,\nabla(2n))
\to 0
\end{equation}
which combined with Theorem \ref{extnabla} gives a recurrence relation
for $e_n=\dim \ext^1_U(k,S^n)$:
\begin{align*}
e_{n}-e_{n-2} &= \begin{cases}
4q  & r=0 \\
2q+3 & r=p-1 \\
0 & \text{otherwise.}
\end{cases} \end{align*}
Together with $e_0=e_1=0$ this determines the
dimensions completely.
\begin{prop}
Let $n=qp+r$ with $0\leq r<p$. Then
\begin{equation}\label{hh1recrel}
\dim \ext^1_U(k,S^n)= \begin{cases}
3\binom{q+2}{2}+4\binom{q+1}{2} & r=p-1 \\
4\binom{q+1}{2}-\binom{q}{2} & r \text{ even, } r\neq p-1\\
3\binom{q+1}{2} & r \text{ odd.}
\end{cases}
\end{equation}
\end{prop}

\begin{lem}
The classes of the cocycles (\ref{genset1}) are
a $Z$-generating set for $\ext^1_U(k,S)$.
\end{lem}
\begin{proof}
The first map in (\ref{ses1}) is multiplication by $c$, so this
follows because, as the proof of Proposition \ref{w0w1} showed, the
images of these cocycles in $\ext^1_U(k,\nabla)$ are a $Z_0$-generating
set.
\end{proof}

It remains to find the relations between these generators.  In the
following we write $\alpha \equiv \beta$ to indicate that $\alpha$ and
$\beta$ differ by a coboundary with values in $S$.
\begin{lem}\label{hh1rels}
\begin{enumerate}
	\item $e^p  F + f^p  E \equiv (1/2)h^p  H$
	\item $2e^p \delta - c^{(p-1)/2} E -h^p 
		\delta_f\equiv 0$.
	\item $2f^p \delta + c^{(p-1)/2}F-h^p \delta_e\equiv 0$.
	\item $c^{(p-1)/2}H +2e^p \delta_e -2f^p 
		\delta_f\equiv 0$.
	\item $c^{(p-1)/2}C-h^p\delta -e^p \delta_e-f^p\delta_f=0$.
	\item $c^{(p+1)/2}\delta - h^p C+e^p F-f^p E=0$.
	\item $2f^p C + h^p F -c^{(p+1)/2} \delta_e-f^p
		H \equiv 0$.
	\item $2e^p C - h^p E - c^{(p+1)/2} \delta_f +
		e^p H\equiv 0$.
\end{enumerate}
\end{lem}
\begin{proof}
    (5) and (6), which assert an actual equality of cocycles,
    are easily checked directly.
\begin{itemize}
    \item[(1)] Let $s=(1/p)(c^p-4(ef)^p -h^{2p})$, which makes sense as an element of
$S$ if we expand $c^p$ and perform the division in an appropriate
$\mathbb{Z}$-form. Then $e^p F+f^p E -(1/2)h^p H$ is equal to $-1/4$
times the coboundary of $1\mapsto s$.
\item[(2), (3)]  Write $(a,b,c)$ for the linear map $\gg
\to S$ sending $e$ to $a$, $h$ to $b$ and $f$ to $c$. Then $2e^p
\delta - c^{(p-1)/2}E -h^p \delta_f$ is $e^p$ times the
cocycle \[\left( \frac{h^p-hc^{(p-1)/2}}{2f}, 0 ,
-\frac{h^p-hc^{(p-1)/2}}{2e}\right).\] This is the coboundary of
$1\mapsto -\sum_{r=0}^{(p-1)/3} h^{p-2r-1}c^r/(2r+1)$. Relation (3) is
similar.
%, which follows from the formula $e\cdot
%h^a c^b=a(h^{a+1}c^b-h^{a-1}c^{b+1}/(2f)$. 
%this really works with the symmetric
%version of \delta
%it must have a nicer expression...
%e \cdot h^a c^b = a(h^{a+1}c^b-h^{a-1}c^{b+1})/(2f) is
%the easy way to check this: the sum telescopes
\item[(4)] The cocycle $c^{(p-1)/2}H-2e^p\delta_e+2f^p\delta_f$ equals
%\begin{equation*}
%\left( \frac{c^p-c^{(p+1)/2}h^{p-1} +
%h^{p+1}c^{(p-1)/2} -h^{2p}}{2f}, 0 ,
%-\frac{c%^p-c^{(p+1)/2}h^{p-1} +
%h^{p+1}c^{(p-1)/2} -h^{2p}}{2e}\right)
%\end{equation*}
\begin{multline*}
-h^p\left(\frac{h^p-hc^{(p-1)/2}}{2f},0,
-\frac{h^p-hc^{(p-1)/2}}{2e} \right) \\ +
c^{(p+1)/2} \left( \frac{c^{(p-1)/2}-h^{p-1}}{2f}
,0,-\frac{c^{(p-1)/2}-h^{p-1}}{2e}\right).
\end{multline*}
The first term is a coboundary as in (2), and the
second is the coboundary of $1 \mapsto \sum_{b=0}^{(p-3)/2} c^b h^{p-2(b+1)}/(2(b+1))$.
\item[(7),(8)] The cocycle $2f^p C + h^p F -c^{(p+1)/2} \delta_e -f^pH$
equals
\begin{equation*}
-f^pc\left( \frac{c^{(p-1)/2}-h^{p-1}}{2f} ,0,
-\frac{c^{(p-1)/2}-h^{p-1}}{2e} \right)
\end{equation*}
which is a coboundary as in (4), and (8) follows similarly.
\qedhere
\end{itemize}
\end{proof}

\begin{thm}
$\ext^1_U(k,S)$ is generated as a $Z$-module by the classes of the
cocycles
$\delta,\delta_e,\delta_f,E,H,F,C$ subject to the relations of Lemma \ref{hh1rels}.
\end{thm}
\begin{proof}
Let $M$ be the graded $Z$-module with these generators and relations.
Since the relations do hold in $\ext^1_U(k,S)$, it is enough to
show that the graded pieces of $M$ have dimensions less than or equal
to those of $\ext^1_U(k,S)$.
  
We first show that 
$\ext^1_U(k,S)$ is generated as a $Z_0$-module by the classes of
\begin{equation}\label{z0gens}
c^i \delta, c^i \delta_e, c^i \delta_f, c^j E, c^j H, c^j F, c^j C :
i\leq \frac{p-1}{2}, j \leq \frac{p-3}{2}
\end{equation}
and relations $c^i R$ for $i\leq (p-3)/2$ where $R$ is the first
relation of Lemma
\ref{hh1rels}. Let $N$ be the $Z_0$ module with these generators and
relations: again, it is enough to show that $N$ has the same Hilbert
series as $\ext^1_U(k,S)$.  As
$Z_0$ is polynomial in three generators of degree $p$, its degree $n$
part has dimension $d_n=\binom{n/p+2}{2}$ if $p|n$ and $0$ otherwise,
so the degree $n$ part of $N$ has dimension
\begin{equation*}
c_n=3 \sum_{i=0}^{(p-1)/2}
d_{n-(p-1)-2i}+4\sum_{i=0}^{(p-1)/2}d_{n-p-2i}-\sum_{i=0}^{(p-3)/2}d_{n-2p-2i}.
\end{equation*}
This obeys the recurrence
relation (\ref{hh1recrel}) and its initial conditions, so the claim
follows. We now need to show that the
graded pieces of $M$ have at most the dimension of those of $N$.

Consider the $Z_0$-submodule $N'$ of $M$ generated by
(\ref{z0gens}). The relations $c^i R$ hold in $M$ so the submodule
$N'$ has graded
pieces of dimension less than or equal to those of $N$. But the
relations (\ref{hh1rels}) tell us $N'=M$: the only $Z_0$-generators
for $N$ missing from $N'$ are those of the form 
$c^i \delta, c^i \delta_e, c^i \delta_f$ with $i \geq (p+1)/2$ and
$c^jE,c^jH,c^jF,c^jC$ with $j \geq (p-1)/2$, and 
relations (6) to (8) show that the former lie in $N'$ while 
relations (2) to (5) show that the latter lie in $N'$.
\end{proof}

\subsection{$\HH^2(U)$}
Propositions \ref{w0w1} and Corollary \ref{w2} show that
\begin{equation*}
0 \to \ext^2_U(k,S^{n-2})\to \ext^2_U(k,S^n) \to
\ext^2_U(k,\nabla(2n)) \to 0
\end{equation*}
is exact except when $n=p-1$ when the second map has image of
codimension 1.  Combined with Theorem \ref{extnabla} this gives a
recurrence relation for  $e_n=\dim \ext^2_U(k,S^n)$:
\begin{align}\label{recrel}
e_{n}-e_{n-2} &= \begin{cases}
3 & r=p-1,q=0 \\
4q+4 & r=p-1, q>0 \\
2q-1 & r=0 \\
0 & \text{otherwise.} 
\end{cases} \end{align}
Since $e_0=e_1=0$ this determines the dimensions completely.
\begin{prop}
Let $n=qp+r$ with $0\leq r<p$. Then
\begin{equation*}
\dim \ext^2_U(k,S^n)= \begin{cases}
3 \binom{q}{2}+\binom{q+1}{2}-\binom{q-1}{2} & r \text{ even}, r \neq
p-1 \\
3\binom{q+2}{2} & r=p-1 \\
3\binom{q+1}{2} & r \text{ odd.}
\end{cases}
\end{equation*}
\end{prop}

We now find the $Z$-module structure of $\ext^2_U(k,S)$.  
With the convention that $(a,b,c)$ denotes the map $\wedge^2\gg \to S$
sending $e\wedge h$ to $a$, $e\wedge f$ to $b$, and $h\wedge f$ to $c$,
the following are cocycles:
\begin{multline*}
    R_e = (f^{p-1},0,0),
    R_h =( eh^{p-2},0,fh^{p-2}),
    R_f =(0 ,0,e^{p-1} ),
    T =(eh^{p-1},0,   fh^{p-1}) .
\end{multline*}
%\begin{enumerate}
%	\item $R_h(h\wedge f)=R_h(e\wedge h)=0, R_h(e\wedge f)=h^{p-1}$.
%	\item $R_e(h\wedge f)=R_e(e\wedge f)=0, R_e(e\wedge h)=f^{p-1}$.  \item $R_f(e\wedge f)=R_f(e \wedge h)=0, R_f(h\wedge
%		f)=e^{p-1}$
%	\item $T(e\wedge f)=0, T(h\wedge f)=fh^{p-1}, T(e\wedge
%		h)=eh^{p-1}$.
%\end{enumerate}
%
There is a relation amongst these generators over $Z$.  In the
following, we again use $\equiv$ to denotes that two cocycles differ by
a coboundary with values in $S$.
\begin{lem}\label{hh2rel}
$e^p R_e + f^p R_f - h^p R_h \equiv c^{(p-1)/2} T$.
\end{lem}
\begin{proof} If $i \neq (p-1)/2$ then in the notation of
    (\ref{2bdry}),
    \begin{equation*}
        \frac{1}{2(2i+1)} t_{c^i  h^{2p-2i-2}}=
    \left(ec^ih^{2p-2i-2},-\frac{c^ih^{2p-2i-1}}{2(2+i)},fc^ih^{2p-2i-2}\right).
    \end{equation*}
    Therefore
    \begin{align*}
        e^pR_e+f^pR_f &= (e(ef)^{p-1},0,f (ef)^{p-1}) \\
                      &=\left(\sum_{i=0}^{p-1}ec^ih^{2p-2i-2} ,0,\sum_{i=0}^{p-1}fc^i
        h^{2p-2i-2}\right) \\
        & \equiv \left( ec^{(p-1)/2}h^{p-1}, \sum_{i=0,i\neq
        \frac{p-1}{2}} ^{p-1}
    \frac{c^ih^{2p-2i-1}}{2(2+i)},fc^{(p-1)/2}h^{p-1} \right)\\
    &= c^{(p-1)/2}T + \left(0,\sum_{i=0,i\neq
        \frac{p-1}{2}} ^{p-1}
    \frac{c^ih^{2p-2i-1}}{2(2+i)},0\right)
\end{align*}
If $i\neq 0$ then $(0,c^ih^{2p-2i-1},0)$ equals the coboundary $s_x$ where 
$x=\frac{1}{i} \sum _{j=0}^{i-1} f c^j h^{2p-2-2j}$. Thus
\begin{equation*}
    e^pR_e+f^pR_f-c^{(p-1)/2}T \equiv (0,h^{2p-1}/2,0).
\end{equation*}
Since 
	$t_{h^{p-1}}=(2eh^{p-2},-h^{p-1},2fh^{p-2})$, 
$h^pR_h$ differs from $(0,h^{2p-1}/2,0)$ by a coboundary and the result follows.
\end{proof}

\begin{thm}
$\ext^2_U(k,S)$ is generated as a $Z$-module by the classes of $R_e,R_f,
R_h$, and $T$ subject to the relation of Lemma \ref{hh2rel}.
\end{thm}
\begin{proof}
  
The images of these cocycles under $\phi$ are
\begin{equation*}
    \begin{array}{lll}
        \phi(R_e)=(y^{2p-2},0,0) & &
\phi(2R_h)=(x^{p}y^{p-2}, 0, -x^{p-2}y^{p})\\
\phi(R_f)=(0,0,x^{2p-2}) & &
\phi(2T)= (x^{p+1}y^{p-1}, 0, -x^{p-1}y^{p+1}).
\end{array}
\end{equation*}
To show that they generate $\ext^2_U(k,S)$ as a $Z$-module,
we need to show that their images
under $\phi$ generate $\phi( \ext^2_U(k,S)) \subset
\ext^2_U(k,\nabla)$ as a $Z_0$-module. This is slightly more
complicated than before as $\phi(\ext^2_U(k,S))$ is not all of
$\ext^2_U(k,\nabla)$.

By Proposition \ref{ext2nabla}, $\ext^2_U(k,\nabla)$ is
generated over $Z_0$ by $\ext^2_U(k,\nabla(2p))$ and
$\ext^2(k,\nabla(2p-2))$.  It follows that $\phi(\ext^2_U(k,S))$ is
generated over $Z_0$ by $\phi( \ext^2_U(k, S^{p-1}))$,
$\ext^2(k,\nabla(2p))$, and $\ext^2(k,\nabla(4p-2))$. We will show
that the images of $R_e,R_h,R_f,T$ are $Z_0$-generators for these.

The proof of Proposition \ref{ext2nabla} and the calculation above shows that
$\phi(T)$ spans $\ext^2_U(k,\nabla(2p))$ and that $\phi(R_e),
\phi(R_h),$ and $\phi(R_f)$ span $\phi(\ext^2_U(k,S^{p-1}))$.   It
also gave a set of cocycles whose images form a basis of
$\ext^2_U(k,\nabla(4p-2))$:
\begin{equation*}
    \begin{array}{ll}
        (x^{ap}y^{(4-a)p-2},0,0), & (0,0,x^{(4-a)p-2} y^{ap})
    \end{array}
\end{equation*}
for $0\leq a \leq 3$. These are in
the span of the cocycles obtained by multiplying
$\phi(R_e),\phi(R_h),\phi(R_f)$ by $x^{2p}, x^py^p, y^{2p}$.

Let $M$ be the graded $Z$-module with the given generators and
relations. Now that we know the classes of $R_e,R_h,R_f$, and $T$ generate
$\ext^2_U(k,S)$ under $Z$ and that the given relation does hold, we
only need to show that the dimensions of the graded pieces of $M$ are
less than or equal to those of $\ext^2_U(k,S)$. 

Since $Z$ is free as a $Z_0$-module on $c^i$ for $0\leq i <p$,
the dimension $f_n$ of the polynomial degree $n$ part of $Z$ is as
follows. Let $n=qp+r$ for $0\leq r<p$. Then
\begin{equation*}
f_n=	\begin{cases}
		{q+2 \choose 2}&  r \text{ even}\\
		{q+1 \choose 2}&  r \text{ odd}
	\end{cases}
\end{equation*}
so the degree $n$ part of $M$ has dimension
$3f_{n-(p-1)}+f_{n-p}-f_{n-(2p-1)}$. This obeys the recurrence
relation (\ref{recrel}) and the same initial conditions. 
\end{proof}

\subsection{$\HH^3(U)$}

Corollary \ref{w2} implies that
\begin{equation*}
0 \to \ext^3_U(k,S^{n-2}) \to \ext^3_U(k,S^n) \to \ext^3_U(k,
\nabla(2n)) \to 0
\end{equation*}
is exact except when $n=p-1$ when the first map has kernel of
dimension 1. Combined with Theorem \ref{extnabla} this gives a
recurrence relation for $e_n = \dim \ext^3_U(k,S^n)$: 
\begin{align*}
e_{n}-e_{n-2} &= \begin{cases}
1 & r=q=0 \\
2q+1 & r=p-1,q>0 \\
0 & \text{otherwise.}
\end{cases}
\end{align*}
Since $e_0=1, e_1=0$ this determines the dimensions completely.
\begin{prop}\label{hh3_dims} Let $n=qp+r$ for $0 \leq r <p$. Then
	\begin{equation*}\dim \ext^3_U(k, S^n) = \begin{cases}
		1 &n \leq p-3 \text{ even}\\
		{q+2 \choose 2} & r=p-1 \\
		{q+1 \choose 2} & r \text{ odd} \\
		{q\choose 2} & r<p-1 \text{ even}.
	\end{cases}\end{equation*}
\end{prop}

It is possible to show, using similar methods to the ones in the
previous two subsections, that $\ext^3_U(k,S)$ is generated by the
classes of the cocycles
\begin{align*}
I(f\wedge h\wedge e) &= 1 \\
J(f\wedge h\wedge e)&= h^{p-1}
\end{align*}
subject to the relations $e^pI\equiv f^pI\equiv h^pI\equiv
c^{(p-1)/2}I\equiv 0$. 
    
\section{Divided powers}
Let $V$ be a $k$-vector space with basis $v_1,\ldots, v_d$.
The divided power algebra on $V$, written $D(V)$, is a
commutative algebra with a basis $v_1^{(a_1)}\cdots v_d^{(a_d)}$ for
$a_i \geq 0$ and multiplication determined by
\begin{equation*}
    v^{(a)}v^{(b)} = \binom{a+b}{a} v^{(a+b)}.
\end{equation*}
$D(V)$ is $\mathbb{Z}$-graded, with the degree $n$ part spanned by
monomials $v_1^{(a_1)}\cdots v_d^{(a_d)}$ with $\sum_i a_i = n$.
See  \cite[A2.4]{E} for more details.  

It turns out that if the symmetric algebra $S(V)$ is given  the coalgebra structure $S(V) \to
S(V)\otimes S(V)$ which sends $v\in V$ to $v\otimes 1+1\otimes v$, the
graded dual $\bigoplus S^n(V)^*$ becomes an algebra isomorphic to
$D(V^*)$.  This  allows us to extend a $\gg$ or $G$-action on $V$ to an
action by derivations or automorphisms on $D(V^*)$.

Let $V=\gg$ and write $D=D(\gg)$ and $D^n = D^n(\gg)$. We have $(D^n)^* \cong
S^n(\gg^*)\cong S^n(\gg)$ as $G$-modules, and so
\begin{equation}\label{DinvExt3}
    \ext^3_U(k,S) \cong \frac{S}{\gg \cdot S} \cong \hom_U(S,k)^* \cong
    \hom_U(k,S^*)^* \cong (D^\gg)^*,
\end{equation}
where $S^*= \bigoplus_{n=0}^\infty (S^n)^*$ denotes the graded dual of $S$. Therefore Proposition
\ref{hh3_dims} determines the dimensions of the graded pieces of the
$\gg$-invariants
on $D$. In this section we will determine the $G$-module
structure of these invariants, and the dimensions of the graded pieces
of the $G$-invariants.

\begin{lem}
$s_1 = e^{(p-1)}h^{(p-1)}f^{(p-1)}$ is $G$-invariant.
\end{lem}
\begin{proof}
This can be shown directly by computing the $G$-action and using Lucas'
theorem on binomial coefficients, or by observing that $s_1$ spans the socle
of the subalgebra of $D$ generated by $e^{(1)}, f^{(1)}$, and $h^{(1)}$
which must be preserved by $G$ since it acts by algebra automorphisms.
\end{proof}

\begin{lem}
Let $s:D\to D$ be $s(x)=s_1x$. Then $\ker s$ is spanned by all
monomials with an exponent not divisible by $p$, and $\im s \cong D^F$
as $G$-modules.
\end{lem}
\begin{proof}
  Lucas' theorem on binomial coefficients implies that if
  $0\leq r,t <p$ then
\begin{equation*}
\binom{ qp+r}{t} \equiv \binom{r}{t} \mod p
\end{equation*}
where the binomial coefficient on the right should be interpreted as 0
if $t>r$.
If any exponent of the monomial $e^{(a)}h^{(b)}f^{(c)}$ is not
divisible by $p$ then the monomial is killed by multiplication by
$s_1$, since for example if $r>0$ then
\begin{equation*}
e^{(p-1)}e^{(qp+r)}= \binom{(q+1)p+r-1}{p-1}e^{((q+1)p+r-1)} = \binom{r-1}{p-1}e^{((q+1)p+r-1)}=0.
\end{equation*}
On the other hand
\begin{align*}
    s_1 e^{(q_ep)} h^{(q_hp)} f^{(q_fp)}&=
    \binom{p-1}{p-1}\binom{p-1}{p-1}\binom{p-1}{p-1}e^{(q_ep+p-1)}
    h^{(q_hp+p-1)} f^{(q_fp+p-1)}\\ &=e^{(q_ep+p-1)} h^{(q_hp+p-1)} f^{(q_fp+p-1)}
\end{align*}
so $\ker s$ is as claimed.

Now $\im s$ is isomorphic to $D/\ker s$ which has a basis consisting of the cosets of
monomials with every exponent divisible by $p$, so is isomorphic to
$D^F$.
%
%the ses is just the dual of S(\gg^p) \to S \to S/S(\gg^p) -- note
%that it really is S(\gg^p) here, not the ideal it generates
%
\end{proof}

\begin{cor}\label{dqf}
    $(D^{(q+2)p+p-3})^{G_1} \cong (D^q)^F$ as $G$-modules.
\end{cor}
\begin{proof}
  The previous lemma tells us that $s_1 D^{qp}$ is a $G$-submodule of
  $D^{(q+2)p+p-3}$ isomorphic to $(D^q)^F$. It is therefore contained
  in $(D^{(q+2)p+p-3})^{G_1}$, but the dimensions are equal by
  (\ref{DinvExt3}) and Proposition \ref{hh3_dims}.
\end{proof}

\begin{lem} \label{dn}
    Let $n=qp+r$ with $0 \leq r<p$. Then unless $r=p-1$ and $q>0$ we have
    $(D^n)^{G_1} \cong (D^{n-2})^{G_1}$ as $G$-modules.
% \begin{enumerate}[a)]
% 	\item If $r\neq 0,p-1$ then $(D^n)^{G_1} \cong (D^{n-2})^{G_1}$
% (this is interpreted as zero if $n=1$).
% \item If $n=qp>0$ then $(D^n)^{G_1} \cong (D^{n-2})^{G_1}$.
% \item If $n=p-1$ then $(D^n)^{G_1}\cong (D^{n-2})^{G_1}$.
% \item If $n=qp + p-1$ for $q>0$ then
% \begin{equation*}
% 	0 \to \Delta(2q)^F \to (D^n)^{G_1} \to (D^{n-2})^{G_1} \to 0.
% \end{equation*}
% \end{enumerate}
\end{lem}
\begin{proof}
By dualising (\ref{SSnabla}) and applying $\hom_{G_1}(k,-)$ we get a
long exact sequence beginning
\begin{equation}\label{les2}
0 \to \Delta(2n)^{G_1} \to (D^n)^{G_1} \to (D^{n-2})^{G_1} \to
\ext^1_{G_1}(k, \Delta(2n)) \to \cdots
\end{equation}
where $\Delta(2n)$ denotes the Weyl module for $G$ corresponding to
the weight
$2n$. \cite[2.3, 2.6]{K} imply that as $G$-modules
\begin{align*} \Delta(2n)^{G_1} &\cong \begin{cases}
k& n=0 \\
\Delta(q-1)^F & 2n=qp + p-2 \\
0 & \text{otherwise}
  \end{cases}\\
  \ext^1_{G_1}(k,\Delta(2n)) &\cong \begin{cases}
k & n=p-1 \\
\Delta(q-2)^F & 2n=qp, q>0 \\
0 & \text{otherwise.}
\end{cases} \end{align*}

For $r\neq 0,p-1$ both the Weyl invariants and ext group in (\ref{les2})
vanish, so the result holds in these cases.  When $r=0$,
(\ref{les2}) becomes
\begin{equation*}
0 \to (D^{qp})^{G_1} \to (D^{qp-2})^{G_1} \to \cdots
\end{equation*}
(\ref{DinvExt3}) and Proposition \ref{hh3_dims} shows that the two
spaces of invariants have the same dimension, so are isomorphic.
Finally
if $n=p-1$ then $D^n \cong S^n$ and $D^{n-3} \cong S^{n-3}$, so the
spaces of invariants are both one-dimensional.
%
% Finally if $n=qp+p-1$ for $q>0$ then (\ref{les2}) is
% \begin{equation*}
% 0 \to \Delta(2q)^F \to (D^n)^{G_1} \to (D^{n-2})^{G_1} \to 0
% \end{equation*}
% as claimed.
\end{proof}

\begin{thm}\label{g1inv}
Let $n=qp+r$ with $0\leq r<p-1$. As $G$-modules,
	\begin{equation*}
		(D^n)^{G_1} \cong \begin{cases}
			k &n\leq p-3 \text{ even,} \\
			(D^q)^F &  r=p-1,\\
(D^{q-1})^F &  r \text{ is odd,} \\
(D^{q-2})^F &  r<p-1 \text{ is even.}
		\end{cases}
	\end{equation*}
\end{thm}
\begin{proof}
    The proof is by induction on $n$, and the inductive step is
    immediate from Lemma \ref{dn} 
    except when it is attempting to prove that $(D^{qp+p-1})^{G_1}$ is
    as claimed for $q>0$.

	% Induction on $n$.  For $n\leq p-3$ we know the dimension is 0 or
% 1 by Proposition \ref{hh3_dims}, so the invariants must be as claimed.

% For $r=2s+1$ odd, $(D^{qp+2s+1})^{G_1} \cong (D^{qp+2s-1})^{G_1}$ by
% Lemma \ref{dn}. If $s>0$ then $2s-1>0$ is odd, so the latter is
% $(D^{q-1})^F$ by the odd residue inductive hypothesis. If $s=0$ then
% \begin{equation*}
	% (D^{qp+1})^{G_1} \cong (D^{(q-1)p + p-1})^{G_1}
% \end{equation*}
% by Lemma \ref{dn}, so by the residue $-1$ inductive hypothesis we get
% $(D^{q-1})^F$ as required.

% For $r=2s<p-1$ even,
% \begin{equation*}
	% (D^{qp+2s})^{G_1} \cong (D^{qp+2s-2})^{G_1}.
% \end{equation*}
% If $s>0$ we're still in the even residue case on the rhs above, with
% the same $q$, so the invariants on the right are $(D^{q-2})^F$ by induction, as required. If $s=0$ then 
% \begin{equation*}
	% (D^{qp})^{G_1} \cong (D^{(q-1)p+p-2})^{G_1} \cong (D^{q-2})^F
% \end{equation*}
% by Lemma \ref{dn} and the odd residue inductive hypothesis, so we're
% done in this case too.

But $(D^{(q+2)p + p-3})^{G_1} \cong (D^q)^F$ by
Lemma \ref{dqf}, and applying Lemma \ref{dn} multiple times,
\begin{multline*}
	(D^{(q+2)p+p-3})^{G_1} \cong (D^{(q+2)p + p-5})^{G_1} \cong
\cdots \cong (D^{(q+2)p})^{G_1} \\ \cong (D^{(q+1)p+p-2})^{G_1} \cong
\cdots \cong (D^{(q+1)p+1})^{G_1} \cong (D^{qp+p-1})^{G_1}
\end{multline*}
so this last space of invariants is isomorphic to $(D^q)^F$ as claimed.
\end{proof}

\begin{cor}
	$(D^n)^G = k$ if $n$ is even and 0 otherwise.
\end{cor}
\begin{proof}
Induct on $n$: for $n\leq p-1$ the divided powers agree with the
symmetric powers, so the result follows because $S^G = k[c]$. 

If $n=qp+r$ is odd then $r$ and $q$ have opposite parity. The
$G$-invariants on $D^n$ are contained in the $G_1$-invariants, and by 
Theorem \ref{g1inv} $(D^n)^{G_1}\cong (D^m)^F$ where $m<n$ is odd. By
the inductive hypothesis
the $G$-invariants in $D^m$ hence in $(D^m)^F$ are zero. 

If $n=qp+r$ is even then $r$ and $q$ have the same parity. This time
$(D^n)^{G_1}\cong (D^m)^F$ where $m<n$ is even, so by
induction the $G$-invariants in $D^m$, and hence $(D^m)^F$, are
one-dimensional.
\end{proof}

\bibliography{HHUsl2}
\bibliographystyle{amsalpha}
\end{document}